\documentclass[a4paper,abstracton,11pt]{scrartcl}
\usepackage{setspace}
\RequirePackage{amssymb}
\usepackage{fullpage}
\usepackage{amsbsy}
\usepackage{enumerate}
\usepackage{mathrsfs}
\usepackage{amsmath}
\usepackage{url}
\usepackage{comment}

\usepackage[utf8]{inputenc}
\usepackage{amsfonts,amsthm,amssymb,amsmath}
\usepackage{color}
\usepackage{graphicx}
\usepackage{authblk}

\theoremstyle{plain}
\newtheorem{thm}{Theorem}
\newtheorem*{thm*}{Theorem}
\newtheorem*{mainthm*}{Main Theorem}
\newtheorem*{mainlem*}{Main Lemma}
\newtheorem{lem}[thm]{Lemma} \newtheorem*{lem*}{Lemma}
 \newtheorem*{claim*}{Claim}
\newtheorem{cor}[thm]{Corollary} \newtheorem*{cor*}{Corollary}
\newtheorem{prop}[thm]{Proposition} \newtheorem*{prop*}{Proposition}

\theoremstyle{definition}
 \newtheorem*{defn*}{Definition}

\theoremstyle{remark}
 \newtheorem*{rem*}{Remark}
 \newtheorem*{example*}{Example}
 \newtheorem*{conj*}{Conjecture}
\newtheorem{question}[thm]{Question} \newtheorem*{question*}{Question}

\title{The distinguishing index of graphs with infinite minimum degree}
\author[1]{Marcin Stawiski}
\author[2]{Trevor M.\ Wilson}

\affil[1]{AGH University of Science and Technology\\ Faculty of Applied Mathematics, \protect\\al. Mickiewicza 30, 30-059 Krakow, Poland}
\affil[2]{Miami University\\Department of Mathematics\\123 Bachelor Hall\\301 S.\ Patterson Ave.\\Oxford, Ohio 45056, USA}
\begin{document}

\maketitle

\begin{abstract}
 The distinguishing index $D'(G)$ of a graph $G$ is the least number of colors necessary to obtain an edge coloring of $G$ that is preserved only by the trivial automorphism. We show that if $G$ is a connected $\alpha$-regular graph for some infinite cardinal $\alpha$ then $D'(G) \le 2$, proving a conjecture of Lehner, Pil\'{s}niak, and Stawiski. We also show that if $G$ is a graph with infinite minimum degree and at most $2^\alpha$ vertices of degree $\alpha$ for every infinite cardinal $\alpha$, then $D'(G) \le 3$. In particular, $D'(G) \le 3$ if $G$ has infinite minimum degree and order at most $2^{\aleph_0}$.
\end{abstract}

In this note we consider undirected graphs without loops or multiple edges. An \emph{edge coloring} of a graph is a function from its set of edges to a set of elements called colors. (We do not require edge colorings to be proper.) We say that an automorphism of a graph $G$ \emph{preserves} an edge coloring of $G$ if it maps every edge to an edge of the same color. The \emph{distinguishing index} $D'(G)$ of a graph $G$ is the least cardinal number $d$ such that $G$ has an edge coloring with $d$ colors that is preserved only by the trivial automorphism. The distinguishing index is defined only for graphs with no isolated edges and at most one isolated vertex; all graphs that we consider will satisfy this property.

Because $D'(G) = 1$ for $G \neq K_1$ simply means that $G$ is asymmetric (lacks nontrivial automorphisms), the main question about the distinguishing index of any given graph $G$ is whether $D'(G) \le 2$, in other words whether the edges of $G$ can be colored with two colors in a way that ``breaks'' every nontrivial automorphism. Note that this problem is related (but not equivalent) to the problem of the existence of an asymmetric spanning subgraph which may be interesting on its own.

The distinguishing index was introduced by Kalinowski and Pil{\'s}niak \cite{kalinowski2015distinguishing}; the corresponding notion for vertex colorings called the \emph{distinguishing number} $D(G)$ had been previously introduced by Babai \cite{babai1977asymmetric} in 1977 under the name \emph{asymmetric colorings}. Distinguishing colorings play an important role in his study of the graph isomorphism problem which includes  the proof of the existence of a quasi-polynomial algorithm for this problem (see \cite{babaiisomorphism}).
Albertson and Collins \cite{albertson1996symmetry} reintroduced the concept of asymmetric colorings in 1996 and coined the term \emph{distinguishing coloring}. 

In this note we will consider the distinguishing index of graphs with infinite minimum degree, possibly uncountable. Although most previous work on the distinguishing index has focused on locally finite graphs, graphs with infinite degrees have also been considered in \cite{babai1977asymmetric,broere2015distinguishing,broere2016distinguishing,lehner2017breaking,lehner2020bound,lehner2020symmetries}. Note that the very first result on the distinguishing colorings is the theorem of Babai \cite{babai1977asymmetric} who showed that every connected tree of infinite degree has a asymmetric spanning forest. In this paper we generalize this result to every connected regular graph of infinite degree.


Broere and Pil{\'s}niak \cite{broere2015distinguishing} showed that $D'(G) \le 2$ for every connected $\aleph_0$-regular graph $G$ and more generally for every countable graph with what they called ``good degrees''. A similar result for much larger graphs was obtained by Lehner, Pil\'{s}niak, and Stawiski \cite[Theorem 5]{lehner2020bound}, who showed that $D'(G) \le 2$ for every connected $\alpha$-regular graph $G$ where $\alpha$ is a fixed point of the aleph sequence of cardinals, meaning $\alpha = \aleph_\alpha$.\footnote{The aleph sequence is the increasing enumeration of all infinite cardinals, beginning with $\aleph_0 = |\mathbb{N}|$.  Further details can be found in a set theory reference such as Jech \cite[Chapter 3]{jec2002}.} They conjectured that $D'(G) \le 2$ holds more generally if $G$ is a connected $\alpha$-regular graph for any infinite cardinal $\alpha$ whatsoever. We will prove this conjecture as Theorem \ref{thm:regular-connected} below. Our method is necessarily quite different from the previous special cases, which relied on the existence of $\alpha$ many cardinals less than $\alpha$. We will also show, in Theorem \ref{thm:inf-min-deg}, that $D'(G) \le 3$ for graphs of infinite minimum degree that are not necessarily regular but satisfy a certain condition on the number of vertices of each degree. Note that a version of Theorem \ref{thm:regular-connected} for (both finite and infinite) connected locally finite graphs of order at least 6 was recently proved by Kwaśny and Stawiski \cite{kwasnystawiski2022}. Combining this result with Theorem \ref{thm:regular-connected} we obtained that the distinguishing index of a connected regular graph of order at least 6 is at most 2.

Before proceeding, we briefly review some basic terminology that is mostly standard. The set of vertices and the set of edges of a graph $G$ will be written as $V(G)$ and $E(G)$ respectively. The \emph{degree} of a vertex is its number of neighbors (adjacent vertices) and the \emph{minimum degree} of a graph is the least degree of any of its vertices. A graph is \emph{$\alpha$-regular}, where $\alpha$ is a cardinal, if each of its vertices has degree $\alpha$. A \emph{forest} is an acyclic graph and a \emph{tree} is a connected forest. A spanning subgraph of $G$ is a subgraph containing all vertices of $G$, and a spanning tree (or forest) is a spanning subgraph that is a tree (or forest).  Since we do not require spanning forests to be maximal, a tree may have a disconnected spanning forest, and a forest may be spanned by a smaller forest.

The following lemma turns out to be useful for investigating the distinguishing index of graphs of infinite minimum degree. Theorem \ref{thm:regular-connected} will follow easily from a special case of it, whereas Theorem \ref{thm:inf-min-deg} will require some additional work.
To prove the lemma we will need a theorem of Babai \cite{babai1977asymmetric}, namely that for every infinite cardinal $\alpha$ there is a family of $2^\alpha$ pairwise non-isomorphic asymmetric trees of order $\alpha$.
Actually Babai  proved more specifically that such a family can be constructed in which each tree has vertices of two prescribed degrees only, such as 1 and $\alpha$. This more specific statement will not be needed here.

\begin{lem}\label{lem:main-lemma}
 Let $(T_i)_{i\in I}$ be a family of trees, each of which is regular of infinite degree, and at most $2^\alpha$ of which are $\alpha$-regular for any particular infinite cardinal $\alpha$.  Then the disjoint union $\bigcup_{i \in I} T_i$ has an asymmetric spanning forest with no isolated vertices.
\end{lem}
\begin{proof}
 It suffices to show that for each infinite cardinal $\alpha$, the disjoint union of all the $\alpha$-regular trees in the family has an asymmetric spanning forest in which each component has order $\alpha$. Then we may take the disjoint union of these asymmetric forests for all infinite cardinals $\alpha$ to obtain the desired asymmetric forest spanning all of $\bigcup_{i \in I} T_i$, since an automorphism cannot interchange components of two different orders.
 
 We may therefore assume without loss of generality that each tree $T_i$ in the family is $\alpha$-regular for the same fixed infinite cardinal $\alpha$, and that the index set $I$ has cardinality at most $2^\alpha$. Under this assumption, our task is to show that the disjoint union $\bigcup_{i \in I} T_i$ has an asymmetric spanning forest in which each component has order $\alpha$.
 
 There is a family of $2^\alpha$ pairwise non-isomorphic asymmetric trees of order $\alpha$ by Babai \cite{babai1977asymmetric}. (For example, there is such a family of trees having vertices of degrees 1 and $\alpha$ only.) Fixing a set $J$ of cardinality $\alpha$ to use as another index set, we have $|I \times J| \le 2^\alpha \cdot \alpha = 2^\alpha$, so there is a family $(A_{i,j})_{i \in I, j \in J}$ of pairwise non-isomorphic asymmetric trees of order $\alpha$. The disjoint union $\bigcup_{i \in I, j \in J} A_{i,j}$ of this family is an asymmetric forest in which each component has order $\alpha$, so it suffices to find a spanning forest of $\bigcup_{i \in I} T_i$ isomorphic to $\bigcup_{i \in I, j \in J} A_{i,j}$. Treating each $i \in I$ separately, we will find a spanning forest of $T_i$ isomorphic to $\bigcup_{j \in J} A_{i,j}$.
 
 Let us consider a fixed index $i \in I$.  Omitting this index from our notation, we can describe our task more simply as follows. Given an $\alpha$-regular tree $T$ and a family $(A_j)_{j \in J}$ of trees of order $\alpha$, where $J$ has cardinality $\alpha$, we want to find a spanning forest of $T$ isomorphic to $\bigcup_{j \in J} A_j$. (The remaining argument will not use the assumption that the trees $A_j$ are pairwise non-isomorphic and asymmetric; they can be any trees of order $\alpha$.)
  
 To show that the forest $F = \bigcup_{j \in J} A_j$ is isomorphic to a spanning forest of $T$, we will add some edges to $F$ to obtain an $\alpha$-regular tree $T'$, which must be isomorphic to $T$ because any two $\alpha$-regular trees are isomorphic. We first write the index set $J$ as a disjoint union of sets $J_n$, $n \in \mathbb{N}$ where $J_0$ is a singleton set $\{j_0\}$ and each $J_{n+1}$ has cardinality $\alpha$. Then we can rewrite the forest $F= \bigcup_{j \in J} A_j$ as the following disjoint union of subforests: $F = \bigcup_{n \in \mathbb{N}} F_n$ where $F_n = \bigcup_{j \in J_n} A_j$.
 
 Note that each subforest $F_n$ has order $\alpha$, and that $F_0$ is just a single tree (namely $A_{j_0}$) whereas each $F_{n+1}$ has $\alpha$ components.  For each $n \in \mathbb{N}$, we may add to $F$ a set $E_n$ of $\alpha$ edges between $F_n$ and $F_{n+1}$ in such a way that
 \begin{enumerate}
  \item  every vertex of $F_{n}$ is incident with $\alpha$ edges in $E_{n}$, and
  \item every component of $F_{n+1}$ is incident with one and only one edge in $E_n$.
 \end{enumerate}
 Let $T'$ be the graph obtained by adding the edges in $\bigcup_{n \in \mathbb{N}} E_n$ to the forest $F$. By (1), the graph $T'$ is $\alpha$-regular.  By the ``only one'' part of (2) and the fact that each $F_n$ is acyclic, $T'$ is acyclic.  By the ``(at least) one'' part of (2) and the fact that $F_0 = A_{j_0}$ is connected, $T'$ is connected.  Therefore $T'$ is an $\alpha$-regular tree spanned by $F$ as desired.
\end{proof}

Our applications of Lemma \ref{lem:main-lemma} will only use the fact that $\bigcup_{i \in I} T_i$ has an asymmetric spanning subgraph with no isolated vertices, and not the fact that this subgraph is a forest.

The lemma can be applied in the special case of a single regular tree $T_i$ to prove the conjecture of Lehner, Pil\'{s}niak, and Stawiski \cite[Conjecture 6]{lehner2020bound}.
The only other ingredient we will need is a result of Andersen and Thomassen \cite[Theorem 2$'$]{andersen1980cover}, which says that for every infinite cardinal $\alpha$, every connected $\alpha$-regular graph has an $\alpha$-regular spanning tree.

\begin{thm}\label{thm:regular-connected}
   Let $G$ be a connected $\alpha$-regular graph where $\alpha$ is an infinite cardinal.  Then
   $D'(G) \le 2$.
\end{thm}
\begin{proof}
 The graph $G$ has an $\alpha$-regular spanning tree, which in turn has an asymmetric spanning forest $F$ by Lemma \ref{lem:main-lemma}. Coloring the edges of $F$ red and the other edges of $G$ blue, we obtain an edge coloring of $G$ that is preserved only by the trivial automorphism.
\end{proof}

The following proposition shows that the hypothesis of connectedness can be removed if we impose a certain limitation on the order, or equivalently for $\alpha$-regular graphs, a limitation on the number of components. 
To see that some such limitation is necessary, letting $\alpha$ be an infinite cardinal and denoting the least cardinal greater than $2^\alpha$ by $(2^{\alpha})^+$, observe that the disjoint union of $(2^{\alpha})^+$ copies of any graph of order $\alpha$ has distinguishing index $(2^\alpha)^+$ because any edge coloring with at most $2^\alpha$ colors admits an automorphism interchanging two components. 

\begin{prop}\label{prop:regular-disconnected}
 Let $G$ be an $\alpha$-regular graph of order at most $2^\alpha$ where $\alpha$ is an infinite cardinal.  Then $D'(G) \le 2$.
\end{prop}
\begin{proof}
 The graph $G$ has at most $2^\alpha$ components and we may choose an $\alpha$-regular spanning tree for each component. Applying Lemma \ref{lem:main-lemma} to this family of trees yields an asymmetric spanning forest $F$ of $G$, and as before we may color the edges of $F$ red and the other edges of $G$ blue to obtain an edge coloring of $G$ that is preserved only by the trivial automorphism.
\end{proof}

Note that in the proposition above we obtained even a stronger result. Namely each $\alpha$-regular graph of order at most $2^\alpha$ where $\alpha$ is an infinite cardinal has an asymmetric spanning forest. It also generalizes the mentioned result of Babai which showed that each infinite regular tree has an asymmetric spanning forest. We showed it for every connected regular graph of infinite degree.

The following theorem shows that the hypothesis of regularity can also be removed if we allow an additional color and we again impose a certain limitation, this time on the number of vertices of each degree. To see that some such limitation is necessary even for connected graphs, observe that the complete bipartite graph $K_{\alpha, (2^\alpha)^+}$ has distinguishing index $(2^\alpha)^+$ because any edge coloring with at most $2^\alpha$ colors admits an automorphism interchanging two vertices of degree $\alpha$.

To prove the theorem we will need another result of Andersen and Thomassen \cite[Theorem 2]{andersen1980cover}, which says that every graph $G$ of infinite minimum degree $\delta$ contains (as a subgraph, not necessarily induced) a $\delta$-regular tree.  We will also need the fact that such a tree may be taken to include any given vertex of $G$,
which is apparent from their proof.

\begin{thm}\label{thm:inf-min-deg}\footnote{We note that this theorem is due to second author solely.}
 Let $G$ be a graph with infinite minimum degree and at most $2^\alpha$ vertices of degree $\alpha$ for every infinite cardinal $\alpha$.  Then $D'(G) \le 3$.
\end{thm}
\begin{proof}
 Let us call a tree $T$ contained in $G$ \emph{degree-realizing} if for some infinite cardinal $\alpha$, $T$ is $\alpha$-regular and some vertex of $T$ has degree $\alpha$ in $G$. Note that $G$ contains at least one degree-realizing tree, because we may choose a vertex whose degree equals the minimum degree $\delta(G)$ of $G$ and then find a $\delta(G)$-regular tree containing that vertex.
 
 Let  $(T_i)_{i \in I}$ be a pairwise disjoint family of degree-realizing trees in $G$ that is maximal in the sense that $G$ contains no other degree-realizing tree disjoint from every $T_i$.  (Such a family exists by Zorn's lemma.) Note that for each infinite cardinal $\alpha$, at most $2^\alpha$ trees in this family are $\alpha$-regular.  This is because different $\alpha$-regular trees in the family are disjoint and degree-realizing, so they must contain different vertices of degree $\alpha$ in $G$, and at most $2^\alpha$ such vertices exist in $G$ by our hypothesis.
 It follows by Lemma \ref{lem:main-lemma} that $\bigcup_{i \in I} T_i$ has an asymmetric spanning forest $F$ with no isolated vertices.
 
 Color all edges of $F$ red and color any other edges of the induced subgraph of $G$ on $V(F)$ blue. Because $F$ might not span $G$, we may need to color some additional edges.  For these edges, we will use the colors green and blue. We will never color any more edges outside of $E(F)$ red and $F$ has no isolated vertices, so any automorphism of $G$ preserving our edge coloring will fix $V(F)$ setwise, and therefore also pointwise since $F$ is asymmetric. This will allow us to use $V(F)$ as a kind of foundation or anchor for the rest of our coloring.
 
 Let $H_0$ be the induced subgraph of $G$ on $V(F)$ and let $c_0$ be the edge coloring of $H_0$ that colors all edges of $F$ red and any other edges of $H_0$ blue, as mentioned previously. We define a \emph{good partial coloring} to be a pair $(H, c)$ where
 \begin{itemize}
  \item  $H$ is an induced subgraph of $G$ containing $H_0$,
  \item every automorphism of $G$ fixing $V(F)$ setwise also fixes $V(H)$ setwise,
  \item $c$ is an edge coloring of $H$ extending $c_0$ and using only green and blue outside $E(F)$, and
  \item every automorphism of $G$ fixing $V(F)$ (and therefore also $V(H)$) setwise and preserving $c$ fixes $V(H)$ pointwise.
 \end{itemize}
 Note that $(H_0, c_0)$ itself is a good partial coloring. Let $\mathcal{P}$ be the poset of all good partial colorings ordered by extension: $(H, c) \le (H', c')$ means $H'$ contains $H$ and $c' \restriction E(H) = c$. It's not hard to check that for any chain in $\mathcal{P}$, its union (in the obvious sense) is also in $\mathcal{P}$. Since the union is an upper bound for the chain, a maximal good partial coloring exists by Zorn's lemma.
 
 Let $(H,c)$ be a maximal good partial coloring. If $H = G$, then $c$ is an edge coloring of the entire graph and we are done: only the edges of $F$ are colored red and $F$ has no isolated vertices, so every automorphism of $G$ preserving $c$ must fix $V(F)$ setwise (in fact pointwise since $F$ is asymmetric) and  must therefore be the identity by the definition of ``good partial coloring''.
 
 Now we suppose that $H \ne G$ and argue toward a contradiction. Let $G \setminus H$ be the induced subgraph of $G$ on $V(G) \setminus V(H)$, let $\delta(G \setminus H)$ be the minimum degree of $G \setminus H$, let $X$ be the set of all vertices of $G \setminus H$ whose degree in $G \setminus H$ equals this minimum degree $\delta(G\setminus H)$, and let $H'$ be the induced subgraph of $G$ with vertex set $V(H') = V(H) \cup X$. We will define an edge coloring $c'$ of $H'$ extending $c$ such that $(H', c')$ is a good partial coloring, contradicting the maximality of $(H,c)$.
 
 Because $(H, c)$ is good, every automorphism of $G$ fixing $V(F)$ setwise fixes $V(H)$ setwise, and because $X$ and $H'$ were defined from $H$ in an invariant way, it follows that every automorphism of $G$ fixing $V(F)$ setwise also fixes $V(H')$ setwise.  Again because $(H,c)$ is good, every automorphism of $G$ fixing $V(F)$ (and therefore also $V(H)$) setwise and preserving $c$ fixes $V(H)$ pointwise. Therefore it will suffice to color the edges between $X$ and $H$ with the colors green and blue in a way to ensure that every color-preserving automorphism of $G$ fixing $V(H)$ pointwise must also fix $X$ pointwise. (Any edges between two vertices in $X$ will not be needed for this symmetry breaking, so let us arbitrarily choose to color them blue rather than green.)
 
 Let $\sim$ be the equivalence relation on $X$ where $u \sim v$ means that $u$ and $v$ have the same set of neighbors in $H$. Letting $Y \subseteq X$ be an equivalence class, note that every automorphism of $G$ fixing $V(H)$ pointwise must fix $Y$ setwise. To ensure that every color-preserving automorphism of $G$ fixing $V(H)$ pointwise must also fix $Y$ pointwise, it suffices to color the edges between $Y$ and $H$ in such a way that any two distinct vertices in $Y$ have different sets of green neighbors in $H$ (meaning neighbors in $H$ joined to them by green edges.) To show this is possible, we must show that $Y$ has at most $2^{\alpha_Y}$ vertices, where $\alpha_Y$ is defined as the number of neighbors in $H$ that each vertex in $Y$ has.
 
 The hypothesis of the theorem guarantees that $G$ has at most $2^{\alpha_Y}$ vertices of degree $\alpha_Y$, so it will suffice to show that every vertex $v \in Y$ has degree $\alpha_Y$ in $G$. If not, then since $\alpha_Y$ is its number of neighbors in $H$ and all degrees in $G$ are infinite, its degree in $G$ must equal its degree in $G \setminus H$, which is the minimum degree $\delta(G \setminus H)$ of $G \setminus H$ since $v \in X$. Letting $T$ be a $\delta(G \setminus H)$-regular tree in $G \setminus H$ containing $v$, which exists by the theorem of Andersen and Thomassen, this tree $T$ therefore counts as degree-realizing in $G$. However, this contradicts the maximality of the pairwise disjoint family $(T_i)_{i \in I}$ because $T$ is contained in $G\setminus H$ whereas each $T_i$ is contained in $F$ and therefore in $H$. This contradiction proves the theorem.
\end{proof}

Theorem \ref{thm:inf-min-deg} may leave some room for improvement.  Although we could look for ways to weaken the hypothesis on $G$, the most obvious question is whether or not three colors are actually needed:

\begin{question}\label{ques:two-colors}
  If $G$ is a graph with infinite minimum degree and at most $2^\alpha$ vertices of degree $\alpha$ for every infinite cardinal $\alpha$, must $D'(G) \le 2$?
\end{question}

Note that a positive answer to this question cannot be obtained by finding an asymmetric spanning forest of $G$ as in the proofs of Theorem \ref{thm:regular-connected} and Proposition \ref{prop:regular-disconnected}. This is because for any two infinite cardinals $\alpha < \beta$, the complete bipartite graph $K_{\alpha,\beta}$ (which satisfies the hypothesis of Theorem \ref{thm:inf-min-deg} if $\beta \le 2^\alpha$) has no asymmetric spanning forest. As remarked by Dirac and Thomassen \cite[p.~409]{dirac1974existence}, any spanning forest of $K_{\alpha,\beta}$ for infinite cardinals $\alpha < \beta$ must have either $\beta$ isolated vertices or $\beta$ leaves ($\beta$ of which must share a common neighbor) and in either case interchanging two of these vertices produces a nontrivial automorphism of the spanning forest.

A special case of Theorem \ref{thm:inf-min-deg} and Question \ref{ques:two-colors} may be worth noting separately:

\begin{cor}\label{cor:order-continuum}
 If $G$ be a graph with infinite minimum degree and order at most $2^{\aleph_0}$, then $D'(G) \le 3$.
\end{cor}

\begin{question}
  If $G$ is a graph with infinite minimum degree and order at most $2^{\aleph_0}$, must $D'(G) \le 2$?
\end{question}

\section*{Acknowledgement}
The authors would like to thank Louis DeBiasio for his helpful comments on a draft of the article.

\bibliographystyle{plain}
\bibliography{distinguishing-index}

\end{document}